\documentclass{amsart}

\usepackage{amsmath}
\usepackage{amssymb}
\usepackage{graphicx}
\usepackage[colorinlistoftodos]{todonotes}
\usepackage[colorlinks=true, allcolors=blue]{hyperref}
\usepackage{tikz}
\usepackage{endnotes}
\usepackage{mathtools}
\usepackage{multirow}
\usepackage{float}
\usepackage{lipsum}
\usepackage{esint}
\usepackage{color}
\usepackage{mathrsfs}

\theoremstyle{definition}

 \newcommand{\comments}[1]{}

\newcommand{\Sp}{\mathbb S}
\newcommand{\R}{\mathbb{R}}

\newcommand{\N}{\mathbb{N}}

\newcommand{\Si}{\mathcal{S}}

\newcommand{\Ord}{\text{Ord}}

\newcommand{\Deltamod}{\Delta_{\text{mod}}}

\newtheorem{theorem}{Theorem}[section]
\newtheorem{question}{Question}[section]

\newtheorem{lemma}[theorem]{Lemma}
\newtheorem{remark}[theorem]{Remark}

\newtheorem{definition}[theorem]{Definition}

\numberwithin{equation}{section}

\title{On the Possible Orders of Harmonic Maps into Euclidean Buildings}

\thanks{
CB supported in part by NSF DMS CAREER-1750254 and by a Simons Fellowship SFI-MPS-SFM-00011728}

\author[Breiner]{Christine Breiner}
\address{Brown University\\
Department of Mathematics\\
Providence, RI}
\email{christine\underline{ }breiner@brown.edu}

\author[Dees]{Ben K. Dees}
\address{Brown University\\
Department of Mathematics\\
Providence, RI}
\email{benjamin\underline{ }dees@brown.edu}

\begin{document}
\maketitle

\begin{abstract}
    We prove a discreteness result for the possible orders of harmonic maps from surfaces to Euclidean buildings; in particular for a building of type $W$ the order is of the form $\frac mk$ where $k$ divides $|W|$.  This generalizes, in the case where the domain has dimension $2$, the ``order gap" of Gromov and Schoen. This result follows by directly analyzing the behavior of homogeneous maps into Euclidean buildings, and then studying a related spherical billiards problem.
\end{abstract}

\section{Introduction}

In this note, we demonstrate discreteness of the spectrum for the order of a harmonic map from a Riemann surface to a Euclidean building. In particular, we characterize the behavior of homogeneous harmonic maps in this setting. More than simply a curiosity, discreteness of the order spectrum has played a crucial role in regularity and rigidity results for harmonic maps in singular (and smooth) settings. 

As a first step in resolving the $p$-adic superrigidity problem for rank 1 symmetric spaces of non-compact type, Gromov and Schoen \cite{gromov-schoen} prove an order gap. For a harmonic map into a locally finite building $X$, they demonstrate that there exists an $\epsilon_X>0$ so that every point in the domain either has order 1 or order at least $1+ \epsilon_X$. This gap leads to improved regularity for order 1 points which, coupled with an induction on dimension argument, allows them to show the singular set has codimension $2$; with this level of control, they prove an integral Bochner type formula to establish rigidity. Sun \cite{sun} proves the order gap and regularity for harmonic maps into $\R$-trees. A key step in the proof of the holomorphic rigidity conjecture for Teichm\"uller space, settled by Daskalopoulos and Mese in \cite{DMholo}, is also to establish an order gap. More recently, in collaboration with Mese, the authors prove the full non-Archimedean superrigidity \cite{bdm} by first establishing an order gap when $X$ is non-locally finite.

Characterization of tangent maps also feed into regularity, uniqueness and uniformization results. Using quantitative stratification techniques, the second author shows \cite{dees} the singular set of a harmonic map into a locally finite building is $(n-2)$-rectifiable. Extending his work, the authors \cite{bd} define the singular stratum for such maps and prove $k$-rectifiability for the $k$-th stratum. Kuwert \cite{kuwert} determines the discrete order spectrum for harmonic maps from Riemann surfaces into non-positively curved cones by completely characterizing homogeneous harmonic maps in this setting. He then uses this characterization to study minimizers in a degree 1 homotopy class between closed surfaces with the target having isolated cone points--minimizers are unique and are precisely the Teichm\"uller map when the target metric is the Teichm\"uller metric. The first author and Mese \cite{BreinerMese} use the order spectrum and the structure of homogenous maps from \cite{kuwert} to define a winding number for harmonic maps into CAT(1) surfaces and ultimately prove a uniformization theorem for CAT(1) spheres.

What we refer to as the \emph{order} first appears in the literature in Almgren's study of $Q$-valued harmonic maps as the \emph{frequency function} (cf. \cite{almgren}). 
 In the case of $Q$-valued harmonic maps from surfaces, De Lellis and Spadaro \cite{dLS} classify the possible orders of harmonic maps--possible orders are of the form $\frac{a}{Q^*}$ where $Q^*\leq Q$; in particular the possible orders form a discrete set for any fixed $Q$. This discreteness leads to explicit fast decay for the order at scale $r$, which in turn can be used to prove uniqueness for tangent maps. As an application, they show that the singular set of such maps is discrete, an improvement over the bound on the Hausdorff dimension (which is known for all dimensions) \cite{almgren,dLS}. For domains of dimension $n\geq 2$, it is known that the singular set is $(n-2)$-rectifiable by \cite{dmsv}; the key montone quantity in this study is, again, the order.

Mese has conjectured that the set of possible orders for harmonic maps from surfaces into Euclidean buildings is discrete; the present note shows that this conjecture holds. Precisely, we have the following.

\begin{theorem}\label{thm:ord-disc}
    Let $S$ be a domain in a Riemann surface, and let $X$ be a (not necessarily locally finite) Euclidean building of type $W$.  If $u:S\to X$ is a harmonic map, and $x_0\in S$, then the order $\Ord^u(x_0)=\frac mk$ for $m,k\in\N$ where $k$ divides $|W|$.  In particular, the possible orders lie in a discrete set determined solely by $W$.

    If $X$ is a rank $1$ Euclidean building (a tree or an $\R$-tree), then in fact $\Ord_u(x_0)=\frac m2$ for $2 \leq m \in \mathbb N$.
\end{theorem}
\noindent Observe that in the setting where $X$ is a locally finite tree, the result here is already well-known from work on the optimal partition problem (cf. \cite{optimal-survey} and the references therein).

The structure of this paper is as follows:

\begin{enumerate}
    \item In section \ref{sec:prelim}, we briefly outline the necessary definitions and terminology for the remainder of the paper. This includes discussions of Euclidean buildings, harmonic maps, and the order functional. We close this section with a useful reduction of Theorem \ref{thm:ord-disc}; cf. Theorem \ref{thm:hom-ord-disc}.
    \item In section \ref{sec:discord} we establish a dichotomy for homogeneous maps $u$ from surfaces into buildings---either (1) $u$ has order $\frac m2$ for $2\leq m\in\N$, or (2) $u$ takes $\Sp^1$ to a curve at constant distance from $u(0)$.
    \item In section \ref{sec:sphbill} we determine the possible orders arising in case (2) of the previous dichotomy, completing the classification of the possible orders of homogeneous maps from surfaces into buildings. In this section, we relate the problem to a problem of spherical billiards.
    \item Finally, section \ref{sec:rem} has a brief discussion of the difficulties involved in extending our work here into higher-dimensional domains, and a natural question that arises in this vein.
\end{enumerate}

In collaboration with Mese, the authors are extending the ideas of this work to show discreteness for the order when $u:\widetilde M \to X$ is a pluriharmonic map and $\widetilde M$ is K\"ahler.

\subsection*{Acknowledgements} The authors are grateful to Misha Lavrov, Ethan Dlugie, and Chikako Mese for productive and enlightening discussions.

\subsection*{Data Availability} We do not analyse or generate any datasets, because our work proceeds within a theoretical and mathematical approach.
\section{Preliminaries}\label{sec:prelim}

\subsection{CAT(0) spaces}\label{CAT0sec}
A CAT(0) space $(X,d)$ is a geodesic space of non-positive curvature, where curvature is defined through triangle comparison with Euclidean triangles. Particular examples of CAT(0) spaces include Euclidean buildings (the principal focus of this note) as well as Hadamard manifolds.  We refer to \cite{bridson-haefliger} for a complete introduction to these spaces.  For our purposes, it will suffice to know that harmonic maps into CAT(0) spaces satisfy certain regularity properties, detailed in \ref{subsec:harmord}.

\subsection{Euclidean \& Spherical Buildings}

In this paper, we use Kleiner and Leeb's notion of Euclidean buildings and hence refer the reader to \cite{kleiner-leeb} for the details.  The equivalence of this notion with the one developed by Tits~\cite{tits} was established by A.~Parreau \cite{parreau}, so there is no loss in this approach (which is more metric in flavor than Tits').  

Let $\mathbb E^N$ be an $N$-dimensional affine space and  $\partial_{Tits} \mathbb E^N\simeq \Sp^{N-1}$ be its Tits boundary.  Denote by  $\rho:\mathrm{Isom}(\mathbb E^N) \to \mathrm{Isom}(\partial_{Tits} \mathbb E^N)$ the canonical homomorphism which assigns to  each affine isometry its rotational part.  An affine Weyl group $W_{\mathrm{aff}}$ is a subgroup of $\mbox{Isom}(\mathbb E^N)$  generated by reflections with finite {\bf reflection group} $W := \rho(W_{\mathrm{aff}}) \subset \mbox{Isom}(\partial_{Tits} \mathbb E^N)$.  The pair  $(\mathbb E^N, W_{\mathrm{aff}})$ is then called a {\bf Euclidean Coxeter complex}.  A {\bf wall} is a hyperplane of $\mathbb E^N$ which is the fixed point set of a reflection in $W_{\mathrm{aff}}$.

A {\bf Euclidean building} $(X,d)$ of type $W$ is a CAT(0) space endowed with a collection of isometries $\{\iota_\alpha:\R^N\to X\}_{\alpha\in\mathcal A}$ which satisfy certain axioms (cf. \cite[Section 4.1.2]{kleiner-leeb} for the precise list). We refer to the integer $N$ as the  {\bf dimension} of $X$.

As a consequence of \cite[Corollary 4.6.2]{kleiner-leeb}, the atlas $\mathcal A$ satisfies the following two properties: 
\begin{enumerate}
\item Every geodesic segment, ray, and line is contained in an image of an isometric embedding of  the collection  (cf.~\cite[EB3]{kleiner-leeb}), and 
\item Two isometric embeddings $\iota_1$ , $\iota_2$ of the  collection are compatible in the sense that $\iota_1^{-1} \circ \iota_2$ is a restriction of an isometry in $W_{\mathrm{aff}}$  (cf.~\cite[EB4]{kleiner-leeb}) 
\end{enumerate}
and $\mathcal A$ is the maximal collection satisfying the above two properties. We call $\iota \in \mathcal A$ a {\bf chart} and its image $A:=\iota(\R^N)$ an {\bf apartment}.

Given geodesics $c_1, c_2$ emanating from a common point $p\in X$, for each $t>0$ consider the triangle $\Delta p c_1(t)c_2(t)$ and its comparison triangle in $\mathbb R^2$ with vertices labeled by $\Delta P C_1(t) C_2(t)$. Let $\widetilde \angle_P(C_1(t),C_2(t))$ denote the angle at $P$. As this quantity is known to be monotone in $t$, we define $\angle_p(c_1,c_2):= \lim_{t \to 0^+} \widetilde \angle_P(C_1(t),C_2(t))$ (cf~\cite[Section 2.1.3]{kleiner-leeb}).
\begin{definition} \label{geodesicgerms}
Two geodesics $c_1, c_2$ emanating from a common point $p\in X$ are said to be equivalent if  $\angle_p(c_1, c_2)=0$.  A {\bf geodesic germ} at $p$ is an equivalence class of geodesics emanating from $p$.  The space of geodesic germs at $p$ along with the distance function defined by $\angle_p$ is a complete metric space by \cite[Lemma 4.2.2]{kleiner-leeb} and defines  the {\bf space of directions} $\Sigma_pX$.  
\end{definition}

\begin{definition} \label{def:tangentcone}
For $p \in X$, the {\bf tangent cone} $(T_p X,d_p)$ is a metric cone over $\Sigma_p X$.  Denote the vertex of $T_{p} X$ by  $\mathsf{O}$.  Any element of $T_{p} X \backslash \mathsf{O}$ can be written as
$([\gamma], t)$
where $[\gamma]$ is a geodesic germ at $p$ and $t \in (0,\infty)$. 
\end{definition}
\begin{remark}\label{rem:deltamod} 
By \cite[Section 4.2.2]{kleiner-leeb}, for a Euclidean building $X$ of type $W$, and $p\in X$:
\begin{enumerate}
 \item   $\Sigma_pX$ is a spherical building modelled on $(\Sp^{N-1}, W)$.
 \item $\Sigma_pX$ consists of copies of $\Deltamod:=\Sp^{N-1}/W$ with disjoint interiors, and there is a canonical map $\theta_\Sigma:\Sigma_pX\to\Deltamod$ which is an isometry on each copy of $\Deltamod$ in $\Sigma_pX$.
   
\end{enumerate}
\end{remark}

\subsection{Harmonic Maps}\label{subsec:harmord}

We will not provide a full exposition of the theory of harmonic maps into CAT(0) spaces and instead refer the interested reader to \cite{gromov-schoen,korevaar-schoen1,korevaar-schoen2} for further details. It will suffice for our purposes to know that for $\Omega$ a Riemannian domain and $(X,d)$ a CAT(0) space, the function space $W^{1,2}(\Omega,X)$ is sensibly defined. For $u\in W^{1,2}(\Omega,X)$, there exists an energy density function $|\nabla u|^2$ and a trace map which permits us to speak meaningfully of the values of $u$ on $\partial\Omega$. Harmonic maps in our context will mean {\em minimizing} harmonic maps; maps which minimize energy when compared to competitors with the same boundary values.

\begin{definition}
    For a map $u\in W^{1,2}(B_r(x),X)$ the {\bf energy of $u$ on $B_r(x)$} is
    \[
   E_u(x,r):=\int_{B_r(x)}|\nabla u(y)|^2dy.
    \]

    We say that a map $u\in W^{1,2}(B_r(x),X)$ is {\bf harmonic} (or energy-minimizing) if for every $v$ so that $u=v$ on $\partial B_r(x)$,
    \[
    E_u(x,r)\leq E_v(x,r).
    \]
    A map $u\in W^{1,2}(\Omega,X)$ is {\bf harmonic} if for every $x\in \Omega$ there exists an $r>0$ such that $u|B_r(x)$ is harmonic by the previous definition.

    We will sometimes omit the subscript and write $E(x,r)$ when the map under discussion is clear.
\end{definition}

Regularity of harmonic maps into Euclidean buildings is established, and we will mention results which will be of import in this note. Let $u:\Omega\to X$ be harmonic, where $X$ is a Euclidean building. Given $x\in\Omega$, if there exists $r>0$ such that $u|B_r(x)$ maps into a single apartment $A$ of $X$, then $u|B_r(x)$ is simply a map into a Euclidean space.  We thus say that such a point is {\bf regular} and call the set $\mathcal{R}(u)$ of such points the {\bf regular set} (of $u$). If $x\notin\mathcal{R}(u)$, we say that $x$ is a {\bf singular point} and write $x\in\Si(u)$.  When $X$ is a Euclidean building (locally finite or not), we have the following bound on the size of the singular set.

\begin{theorem}\label{thm:haus-bound}[\cite[Theorem 6.4]{gromov-schoen} and \cite[Theorem 1.1]{bdm}]
If $u:\Omega\to X$ is a harmonic map from a Riemannian domain to a Euclidean building, then the singular set $\Si(u)$ has Hausdorff codimension $2$. In particular, if $\dim(\Omega)=2$, $\Si(u)$ has Hausdorff dimension $0$.
\end{theorem}
When $X$ is locally finite, the authors go further and establish in \cite{bd} that the singular set admits a natural stratification into rectifiable strata.
 
\subsection{Order function}\label{subsec:order}

The order $\Ord^u(x_0)$ of a harmonic map $u$ is analogous to the degree of the first nonvanishing term of the Taylor expansion of $u(x)-u(x_0)$.  However, unlike in the smooth setting, $\Ord^u(x_0)$ does not need to be integer valued at singular points of $u$. 

\begin{definition}\label{def:terminology}For a map $u:\Omega\subset\R^n\to X$, where $X$ is a Euclidean building, let
\begin{align}
\nonumber I(x,r)&:=\int_{\partial B_r(x)}d^2(u(y),u(x))dy,\\
\nonumber\Ord(x,r)&:=\frac{rE(x,r)}{I(x,r)}.
\end{align}
\end{definition}

In \cite{gromov-schoen} the authors show that for locally finite targets, there is a $c$ depending only on the metric $g$ of the domain so that $e^{cr^2}\Ord(x,r)$ is a monotonically increasing function of $r$, bounded from below by $1$. This monotonicity extends easily to the non-locally finite setting and thus the limit as $r\to0$ of $\Ord(x,r)$ is hence well-defined for each $x\in\Omega$. 

\begin{definition}
    The {\bf order function} $\Ord:\Omega \to [1,\infty)$ is given by
    \[
    \Ord(x):=\lim_{r\to0}\Ord(x,r).
    \]   
\end{definition}

For harmonic maps into locally finite buildings, \cite{gromov-schoen} establish an order gap.\footnote{In fact, Gromov and Schoen's methods apply to the slightly more general class of $F$-connected complexes.  We will not have cause to work with this class in this note, so we state the result only for Euclidean buildings.} This result is generalized in \cite{sun,bdm} to the non-locally finite setting.

\begin{theorem}\label{thm:gs6.3}[\cite[Theorem 6.3]{gromov-schoen},\cite[Theorem 1.1]{sun}, and \cite[Theorem 6.6]{bdm}]
    For $X$ a Euclidean building, and any $n\in\N$, there is an $\epsilon=\epsilon(X,n)$ so that if $u:\Omega\to X$ is a harmonic map from a Riemannian domain $\Omega$ of dimension $n$, then for all $x\in\Omega$, either $\Ord^u(x)=1$, or $\Ord^u(x)\geq1+\epsilon$.
\end{theorem}

Note that the above is a special case of Theorem \ref{thm:ord-disc}, showing that $1$ is an isolated point in the set of possible orders.

\subsection{Homogeneity and tangent maps}
Gromov and Schoen \cite{gromov-schoen} develop a notion of tangent maps for harmonic maps into some CAT(0) spaces, including locally finite Euclidean buildings.  The situation for more general CAT(0) space targets is developed in \cite{korevaar-schoen1,korevaar-schoen2}, and a particular extension focused on non-locally finite Euclidean buildings can be found in \cite{bdm}. The notion of tangent maps has also been extended to CAT(1) targets in \cite{BFHMSZ}. Tangent maps in this setting are homogeneous but, in contrast to the smooth setting, are non-constant even at smooth points. In particular, at points of order $1$, the tangent map is a non-constant affine map. 

Let $u:B_\sigma(x) \subset \mathbb R^n \to (X,d)$ be a harmonic map into a Euclidean building $X$ and change coordinates in the domain so that $x=0$. Then for each $\lambda>0$ define the map 
\begin{equation}\label{eq:rescalings}
u_{\lambda}:B_{\sigma/\lambda}(0) \to (X,d_\lambda)
\end{equation}
by $u_\lambda(y) := u(\lambda y)$, where $d_\lambda =(\lambda^{1-n}I(0,\lambda))^{-1/2}d$. 

The scalings on domain and target are chosen so that for all $\lambda>0$ the functions $u_\lambda$ have uniform energy bounds on $B_1$.  In the locally finite setting, the regularity theory of \cite{gromov-schoen} implies that a subsequence converges uniformly to a non-constant, harmonic map $u_{*}:B_1 \to T_{u(x)}X$.  In the non-locally finite setting, \cite{bdm} shows how to use ultralimits to obtain a map $u_*$ into an ultralimit building $X_*$. 
\begin{definition}\label{def:tangentmap}
    Any subsequential limit map $u_{*}:B_1 \to T_{u(x)}X$ in the locally finite setting or any ultralimit map $u_*:B_1 \to X_*$ in the non-locally finite setting realized by the rescaling process above is called a {\bf tangent map to $u$ at $x$}.
\end{definition}
\begin{remark}\label{rem:properties}Two crucial properties of this rescaling process should be noted. First, the order remains unchanged. That is
\begin{equation}\label{eq:order_compare}
\Ord^u(x)=\Ord^{u_*}(0).
\end{equation}
Second, if $X$ is of type $W$, then $T_{u(x)}X, X_*$ are both of type $W$ (cf. \cite{kleiner-leeb,bdm}).
\end{remark}

Even in the non-locally finite setting, we shall see that for maps from surfaces, these ultralimits map (at least locally) into a metric cone over a spherical building; cf. Lemma \ref{lem:fin-apt}. We shall need the following definition to establish this.

\begin{definition}
    A harmonic map $u:B_1(0)\to X$ is {\bf intrinsically homogeneous of order $\alpha$} if for all $x\in B_1(0)$, $\lambda\in(0,1)$,
    \begin{enumerate}
        \item 
$        d(u(\lambda x),u(0))=\lambda^\alpha d(u(x),u(0)),
 $     
        \item The image of $\lambda\mapsto u(\lambda x)$ is a geodesic segment from $u(0)$.
    \end{enumerate}

    By \cite[Remark 2.11]{bdm} and \cite[Lemma 3.2]{gromov-schoen}, any tangent map $u_*$ arising from our rescaling process is necessarily intrinsically homogeneous, even if the target is non-locally finite.
\end{definition}

Now we prove that every intrinsically homogeneous map from a Euclidean disk into a Euclidean building maps into a cone over the space of directions, at least locally. This will justify our subsequent focus on {\em conical} buildings, as we can always reduce to this case.

\begin{lemma}\label{lem:fin-apt}
    If $u:B_2(0)\to X$ is an intrinsically homogeneous harmonic map from the unit disc $B_2(0)\subseteq\R^2$ into a non-locally finite Euclidean building $X$, there is a finite collection of apartments $\{A_1,A_2,\dots,A_M\}$ containing $u(0)$ so that $u(B_{1}(0))\subseteq\bigcup_{i=1}^MA_i$. 
    
    In particular, there is a ball $B_\rho(0)$ so that $u(B_\rho(0))$ is isometric to a subset of a finite simplicial complex in $T_{u(0)}X$.
\end{lemma}
\begin{proof}
    Since $u$ is only intrinsically homogeneous, we cannot immediately conclude that the singular set is conical and thus there may be singular points on $\partial B_1(0)$. Nevertheless, we claim that for each $p\in\partial B_{1}(0)=\Sp^1$, there is an open interval $I_p\subseteq\Sp^1$ containing $p$, and a finite collection of apartments $\mathcal{A}_p=\{A_{1,p},\dots,A_{k_p,p}\}$ so that $u(I_p)\subseteq\bigcup_{n=1}^{k_p}A_{n,p}$. If $p$ is a regular point this is of course immediate, with $\mathcal A_p= \{A_{1,p}\}$.

         If $p$ is {\em not} regular, we observe that because the singular set of $u$ has dimension $0$, there is some $\lambda\in(1,2)$ so that $\lambda p$ is regular. Thus there exists an interval $J_{\lambda p} \subseteq \lambda \Sp^1$ with $u(J_{\lambda p})$ contained in a single apartment $A_{\lambda p}$. By intrinsic homogeneity, the set $\hat J_{\lambda p}:=\{u(r,J_{\lambda p}),r \in (0,\lambda]\}$ lies in the convex hull of $u(0)$ and $u(J_{\lambda p})$. Applying \cite[Corollary 4.6.8]{kleiner-leeb}, we obtain a finite collection of apartments $\mathcal{A}_{\lambda p}$ containing $u(0)$ which cover $\hat J_{\lambda p}$.  Define $I_p:= \lambda^{-1}J_{\lambda p}\subseteq \Sp^1$. Then $u(I_p)=u(\lambda^{-1}J_{\lambda p})\subseteq  \hat J_{\lambda p}$, so $u(I_p)$ is covered by the apartments in $\mathcal A_p:= \mathcal{A}_{\lambda p}$.
    
    Now, by compactness, there is a finite collection $\{p_1,\dots,p_K\}\subset\Sp^1$ so that $\{I_{p_j}\}_{j=1}^K$ covers $\Sp^1$. Again, by intrinsic homogeneity we know that $u(B_1(0))$ is contained in the convex hull of $u(0)$ and the apartments in the collections $\mathcal A_{p_1}, \dots, \mathcal A_{p_K}$. And again applying \cite[Corollary 4.6.8]{kleiner-leeb}, we obtain a finite collection of apartments $A_1,\dots,A_M$ such that
    \[
    u( B_1(0))\subseteq\bigcup_{i=1}^MA_i.
    \]

    Now, we observe that the complex $C=\bigcup_{i=1}^MA_i$ is a finite simplicial complex in $X$. There is hence some $r>0$ so that $B_r(u(0))\cap C$ is isometric to a neighborhood in the tangent cone $T_{u(0)}C$. Now, because $T_{u(0)}C\subseteq T_{u(0)}X$, and $u$ is continuous, there must be a $B_\rho(0)$ so that $u(B_\rho(0))$ is isometric to a subset of $T_{u(0)}X$.
    
\end{proof}

Because $T_{u(0)}X$ is a metric cone over $\Sigma_{u(0)}X$, the above result means that tangent maps {into non-locally finite buildings} can be regarded as maps into conical buildings---Euclidean buildings which are metric cones. Hence, by scaling with respect to the cone point, we can give an extrinsic definition for homogeneity as in \cite{bd}, rather than the intrinsic one used above and in \cite{gromov-schoen}.

\begin{definition}\label{def:ext-hom}
    Let $X_C$ be a conical Euclidean building with cone point $0_X$.
    We say that a map $h:\mathbb R^n\to X_C$ is {\bf homogeneous of order $\alpha$ about $x_0$} if $h(x_0)=0_X$ and for all $z\in\mathbb R^n$ and $\lambda>0$
\begin{equation}\label{eq:hom-def}
h(x_0+\lambda z)= \lambda^\alpha h(x_0+z).
\end{equation}
When $x_0=0$ we simply say $h$ is homogeneous of order $\alpha$.
\end{definition}

In \eqref{eq:hom-def}, we interpret the scalar multiplication on the right-hand side to mean scaling the point $h(x_0+z)$ by a factor of $\lambda^\alpha$ about the cone point $0_X$.  (Formally, it is the unique point at distance $\lambda^\alpha d(0_X,h(x_0+z))$ along the geodesic ray from $0_X$ to $h(x_0+z)$, where the uniqueness of this geodesic ray is a consequence of conicality.)

The above discussion of tangent maps means that to classify the possible orders for harmonic maps $u:\Omega\to X$, it suffices to study the orders of homogeneous harmonic maps, because we can reduce to studying the tangent map. (Recall also Remark \ref{rem:properties}.) Precisely, we have the following lemma.

\begin{lemma}\label{lem:reduce-to-hom}
    Suppose that $u:S\to X$ is a nonconstant harmonic map from a domain in a Riemann surface to a Euclidean building $X$ of type $W$.  If $\alpha=\Ord^u(x_0)$ for $x_0\in S$, then there is a nonconstant homogeneous harmonic map $u_*:B_1(0)\to X'$ of order $\alpha$, where $B_1(0)\subset\R^2$ and $X'$ is a conical Euclidean building of type $W$. Additionally, the image of $u_*$ lies in a finite simplicial cone $C\subseteq X'$.
\end{lemma}

In light of this, to prove Theorem \ref{thm:ord-disc}, it suffices to prove the following.

\begin{theorem}\label{thm:hom-ord-disc}
    If $u:B_1(0)\subset\R^2\to X$ is a nonconstant homogeneous harmonic map of order $\alpha$, and $X$ is a conical Euclidean building of type $W$, then $\alpha=\frac mk$ for $m,k\in\N$ where $k$ divides $|W|$.  If $X$ is dimension $1$ (a tree or $\R$-tree), then in fact $\alpha=\frac m2$ for some natural number $m\geq2$.
\end{theorem}

\section{First steps towards discreteness}\label{sec:discord}

In this section, we establish an important dichotomy for homogeneous harmonic maps from surfaces.

\begin{theorem}\label{thm:flat-dichotomy}
    For a nonconstant homogeneous harmonic map $u:B_2(0)\subset\R^2\to X$ of order $\alpha$, where $X$ is a conical Euclidean building, one of the following holds:
    \begin{enumerate}
        \item either $\alpha=\Ord(0)=\frac{m}{2}$ for some $2\leq m\in\N$, or
        \item $d_X(u(\phi),u(0))$ is constant for $\phi\in\mathbb{S}^1$.
    \end{enumerate}
\end{theorem}

\begin{proof}
   By homogeneity, it will suffice to understand the behavior of $u|\partial B_{1}(0)$ (or indeed any circle enclosing the origin). Further, homogeneity and Theorem \ref{thm:haus-bound} imply that $\mathcal S(u) \subset \{0\}$.
    Thus, for each $x\in \partial B_{1}(0)$, there is a ball $B_{\rho(x)}(x)$ on which $u$ is regular, that is, so that $u(B_{\rho(x)}(x))\subseteq A$ for some apartment $A$ of $X$. Because $X$ is conical, we may choose to fix coordinates on {\em all} apartments of $X$ so that $0\in A$ is the cone point $0_X$ in $X$. Working in such coordinates, and regarding $w=u|B_{\rho(x)}(x):B_{\rho(x)}(x)\to A$ as a regular harmonic map, $w$ satisfies the harmonic map equation in polar coordinates:
    \[
    0=\frac1{r^2}\frac{\partial^2w}{\partial\theta^2}+\frac1r\frac{\partial w}{\partial r}+\frac{\partial^2w}{\partial r^2}.
    \]
    Writing $w(r,\theta)=r^\alpha g(\theta)$ and setting $r=1$, we obtain a differential equation for $g$,
    \begin{equation}\label{eq:g-diff-eq}
    0=\frac{\partial^2g}{\partial\theta^2}(\theta)+\alpha^2g(\theta)
    \end{equation}
    so that $g(\theta)={\bf v}_1\cos(\alpha\theta)+{\bf v}_2\sin(\alpha\theta)$ for some vectors ${\bf v}_1,{\bf v}_2\in A$, where {\em a priori} this representation $g$ is valid only on a neighborhood of the point $x=(1,\theta_x)$, possibly depending on the choice of the point $x\in\partial B_1(0)$ and the apartment $A$. To analyze the original map $u$ more closely, we want to determine which features of $g(\theta)$ actually depend on these choices.

    The quantity $q(\theta):=d^2(u(e^{i\theta}),u(0))$ is a well-defined function on $\mathbb S^1$ without reference to $A$ or $x$, but we can also locally express it in terms of $g$, as $q(\theta)=\|g(\theta)\|^2$.  Hence for each $\theta\in\Sp^1$, we may locally compute $q'(\theta)$ using $g$ as follows
    \[
    q'(\theta)=2g(\theta)\cdot g'(\theta)=\alpha(\|{\bf v}_2\|^2-\|{\bf v}_1\|^2)\sin(2\alpha\theta)-2\alpha({\bf v}_1\cdot{\bf v}_2)\cos(2\alpha\theta).
    \]
    This representation reveals that $q(\theta)$ is differentiable (and indeed smooth) at each $x\in\Sp^1$.  Moreover, if $A_1,A_2$ are two apartments so that for an interval $I\subseteq\Sp^1$, $u(I)\subset A_1\cap A_2$, we may compare the representations of $u$ in these apartments.  Suppose that for all $\theta\in I$, 
    \[    u(\theta)={\bf v}_1\cos(\alpha\theta)+{\bf v}_2\sin(\alpha\theta)={\bf w}_1\cos(\alpha\theta)+{\bf w}_2\sin(\alpha\theta)
    \]
    where ${\bf v}_1,{\bf v}_2\in A_1$ and ${\bf w}_1,{\bf w}_2\in A_2$.  Computing $q'(\theta)$ for $\theta\in I$, and comparing coefficients on $\cos(2\alpha\theta)$ and $\sin(2\alpha\theta)$, we conclude that
    \begin{align*}
    \|{\bf v}_2\|^2-\|{\bf v}_1\|^2&=\|{\bf w}_2\|^2-\|{\bf w}_1\|^2,\\
    {\bf v}_1\cdot{\bf v}_2&={\bf w}_1\cdot{\bf w}_2.
    \end{align*}
    That is, that the quantities $\|{\bf v}_2\|-\|{\bf v}_1\|$ and ${\bf v}_1\cdot{\bf v}_2$ are independent of the particular representation $g$ and thus are independent of our original choice of $x$.

    Thus, the representation  
    \[
    q'(\theta)=\alpha(\|{\bf v}_2\|^2-\|{\bf v}_1\|^2)\sin(2\alpha\theta)-2\alpha({\bf v}_1\cdot{\bf v}_2)\cos(2\alpha\theta)
    \]
    is valid for all $\theta\in\Sp^1$.  Now, since $q$ is a smooth function on all of $\mathbb S^1$, unless $\|{\bf v}_2\|^2-\|{\bf v}_1\|^2={\bf v}_1\cdot{\bf v}_2=0$, we conclude immediately that $2\alpha$ is an integer.  On the other hand, if $\|{\bf v}_2\|^2-\|{\bf v}_1\|^2={\bf v}_1\cdot{\bf v}_2=0$, we directly see that $q'(\theta)=0$ for all $\theta$, so that $q(\theta)$ is constant.
\end{proof}

In case (1) of \ref{thm:flat-dichotomy}, $u$ already satisfies the desired discreteness result for the order, so we turn to case (2). To close out this section, we show that when case (2) holds, the {\em order} of a homogeneous map $u$ is closely related to the {\em length} of $u(\Sp^1)$.

\begin{lemma}\label{lem:hom-len-ord}
    Suppose that $u: B_1(0)\subset \R^2\to X$ is a homogeneous harmonic map of order $\alpha$, where $X$ is a conical Euclidean building, and $u(\Sp^1)$ is a curve at constant distance $L$ from $u(0)$.  Then $\ell(u(\Sp^1))=2\pi\alpha L$.
\end{lemma}
\begin{proof}
    It suffices to show that for any point $x_0\in\Sp^1$, there is an interval $I=[a,b]\subset\Sp^1$ containing $x_0$ so that $u(I)$ is an interval of length $\ell(I)\alpha L$.

    To see this, recall that any such $x_0$ is a regular point of $u$, so there is a neighborhood $V$ of $x_0$ and an apartment $A$ so that $u$ has the representation
    \[
    u(r,\theta)=r^\alpha\sin(\alpha\theta)\bf{v}_1+r^\alpha\cos(\alpha\theta)\bf{v}_2
    \]
    where $\bf{v}_1$ and $\bf{v}_2$ are orthogonal vectors of length $L$ in $A$.  If $I\subseteq\Sp^1\cap V$ is an interval of length $\ell(I)$ containing $x_0$, direct computation of the arc length reveals that $u(I)$ is a curve of length $\ell(I)\alpha L$.
\end{proof}

\section{Spherical Billiards}\label{sec:sphbill}

In this section, we connect the problem of finding the length of $u(\Sp^1)$ to a spherical billiards problem in $\Deltamod$ (cf. Remark \ref{rem:deltamod}). First, we show that case (2) of the previous section's dichotomy implies that $u(\Sp^1)\subseteq\Sigma_{u(0)}X$.

\begin{lemma}\label{lem:finreduct}
    If $u: B_2(0)\subset \R^2\to X$ is a homogeneous harmonic map so that $u(\Sp^1)$ is at a constant distance to $u(0)$, there exist $\rho, \beta>0$ so that $ \beta  u(\rho\Sp^1)$ maps isometrically into $\Sigma_{u(0)}X$, the space of directions at $u(0)$.

\end{lemma}
\begin{proof}
By Lemma \ref{lem:fin-apt}, or by local finiteness, there exists a $\rho>0$ so that $u(B_\rho(0))$ is isometric to a subset of $T_{u(0)}X$. Thus by homogeneity and the fact that $u(\Sp^1)$ is at constant distance to $u(0)$, $d(u(0), u(\rho/2,\theta))=\lambda$ for some constant $\lambda>0$. Thus $u(\rho/2,\theta) =(\lambda,[\gamma_\theta])$ for germs $[\gamma_\theta]\in\Sigma_{u(0)}X$ and we may consider $\frac 1 \lambda u|\partial B_\rho(0)$ as a map into $\Sigma_{u(x_0)}X$.
\end{proof}

The previous lemma implies that if $u$ is a homogeneous harmonic map with $d(u(\Sp^1),u(0))$ a constant, then after possibly rescaling the domain and target, we may regard $u|\Sp^1$ as a map into $\Sigma_{u(0)}X$. We now consider properties of this map.

\begin{lemma}\label{lem:gammadef}
Let $\rho,\beta>0$ be as in Lemma \ref{lem:finreduct}.    
    For $\gamma:\Sp^1 \to \Sigma_{u(x_0)}X$ such that 
    \begin{equation}
    \gamma(\theta):=\beta u(\rho, \theta) ,
    \end{equation}
   the following hold:
   \begin{enumerate}
   \item $\gamma$ is a local geodesic with constant speed $\alpha$.
   \item $\ell(\gamma(\Sp^1))=2\pi \alpha$.
   \end{enumerate}
\end{lemma}

\begin{proof}
Given $x \in \Sp^1$, there exists $\delta>0$ and an apartment $A_x$ of $X$ such that $u(B_\delta(x)) \subset A_x$. By homogeneity, $\beta u( B_\rho(0))\cap A_x$ lies in a plane containing $0_X$, so $\beta u(\rho \Sp^1)\cap A_x$ lies in a great circle of $\Sigma_{u(0)}A_x\simeq \Sp^{N-1}$. Thus, $\gamma|(B_\delta(x)\cap \Sp^1)$ maps into $\Sigma_{u(0)}X$ as a geodesic. Since this is true for all $x\in \Sp^1$, the map $\gamma$ is a {local geodesic}. Moreover, in local coordinates within each apartment we may orthonormal vectors $\bf{v}_1,\bf{v}_2$, such that in a neighborhood of $x \in \Sp^1$
\[
\gamma(\theta)= \beta \rho^\alpha(\cos(\alpha \theta) {\bf v}_1+ \sin(\alpha \theta) {\bf v}_2).
\] Noting that $\beta \rho^\alpha =1$ by construction, we see that $\gamma$ has constant speed $\alpha.$ Clearly, the length of its image must be $2\pi \alpha$.
\end{proof}
Presume $X$ is a building of type $W$. Let $\theta_{\Sigma}:\Sigma_{u(0)}X\to\Sp^{N-1}/W=:\Deltamod$, denote the canonical anisotropy map, analogous to the natural map $\overline \theta:\Sp^{N-1}\to\Deltamod$ which quotients out by the action of $W$.  
Consider the behavior of the map 
\[
\overline \gamma:=\theta_\Sigma\circ \gamma \circ t: [0,2\pi \alpha] \to \Deltamod,
\]where $t:[0,2\pi \alpha] \to [0,2\pi]$ such that $t(\theta):= \theta/\alpha$. While in the interior of $\Deltamod$, $\overline\gamma$ traces out a geodesic of unit speed, and when it is incident to a wall in a transverse direction, it reflects off with angle of incidence equal to angle of return.  Notice also that $\overline \gamma(0)= \overline \gamma(2\pi\alpha)$ so we may consider $\overline \gamma:\R \to \Deltamod$ as the obvious periodic map.

We now consider a wider class of maps with the same domain, target, and incidence behavior.

\begin{definition}
    We say that $c:\R\to\Deltamod$ is a {\bf closed billiards path of period $\lambda>0$ in $\Deltamod$} if:
    \begin{enumerate}
        \item For each $x\in\R$, $c$ is a ``unit-speed local geodesic" at $x$ in the sense that there is an interval $I\ni x$ and a unit speed geodesic $\widetilde{c}:I\to\Sp^{N-1}$ so that $c|I=\overline\theta\circ\widetilde{c}$, and
        \item $c$ is $\lambda$-periodic in the sense that
        \[
        c(x)=c(x+\lambda)
        \]and there does not exist $\overline \lambda \in (0,\lambda)$ such that $c(x)= c(x + \overline \lambda)$.
    \end{enumerate}
\end{definition}
\begin{remark}
    Condition (1) can also be phrased (slightly less generally) as the condition that $\gamma$ moves along geodesics on the interior of $\Deltamod$, and reflects off of the walls when it meets them.
\end{remark}
\begin{remark}
        A path satisfying only condition (1) of this definition may reasonably be called a {\bf billiards path in $\Deltamod$}, but in fact every billiards path in $\Deltamod$ is closed; the distinction is relevant in more irregular spherical domains.
\end{remark}
Because the extension of $\overline \gamma$ is a closed billiards path which is $\frac{2\pi\alpha}m$-periodic for some $m \in \mathbb N$, to classify the possible orders of harmonic maps into $(X,W)$, it suffices to classify the possible values of $\lambda$.

\begin{lemma}\label{lem:sph-bill}
    If $c$ is a closed billiards path of period $\lambda$, then there exist $j,k \in \mathbb N$ such that $\lambda = \frac{2\pi j}{k}$ where $k$ divides $|W|$.
\end{lemma}
\begin{proof}
Covering the sphere $\Sp^{N-1}$ by $|W|$ copies of $\Deltamod$, we lift $c$ to $(\Sp^{N-1},W)$. Formally this amounts to finding a geodesic $\overline c:\R \to \Sp^{N-1}$ such that $ c:= \overline \theta \circ\overline c$. Note that while many such lifts exist, prescribing $\overline c(0)$ completely determines the lifted map.

Choose one such lift $\overline{c}$ and note that because $c(0)=c(\lambda)$, there is some $w\in W$ so that $\overline{c}(\lambda)=w\cdot\overline{c}(0)$. Let $\widetilde c(s):=\overline c(s + \lambda)$. Then $\widetilde c, w \cdot \overline c$ are both lifts of $c$ and $\widetilde c(0) = w \cdot \overline c(0)$. It follows that $\widetilde c = w \cdot \overline c$ and thus $w \cdot \overline c([0,\lambda])= \overline c([\lambda,2\lambda])$.

Moreover, $\overline c(2\lambda)= w \cdot \overline c(\lambda)$. Let $k$ be the order of $w$. Then, repeating this process $k$ times we find that $\overline c(k\lambda)= w^k \cdot \overline c(0)= \overline c(0)$. That is, $\overline c:[0,k\lambda]\to \Sp^{N-1}$ is a closed geodesic of unit speed (with multiplicity $j$). Therefore, $k \lambda = 2\pi j$ where $k$ divides $|W|$.
\end{proof}

We now prove Theorem \ref{thm:hom-ord-disc}, completing the proof of Theorem \ref{thm:ord-disc}.

\begin{proof}[Proof of Theorem \ref{thm:hom-ord-disc}]

    Let $u: B_1(0)\to X$ be a nonconstant homogeneous harmonic map of order $\alpha$.  By Theorem \ref{thm:flat-dichotomy} one of the following holds:
    \begin{enumerate}
        \item $\alpha=\frac{m}{2}$ for some $m\in\N$, $m\geq2$.  Because $W$ is a reflection group, $2$ divides $|W|$, so in this case we are done.
        \item There is a constant $c$ so that for any $\phi\in\Sp^1$, $d(u(0),u(\phi))=c$.
    \end{enumerate}
    
We now consider only case (2). Assume first that $N=1$. Since the image of $u$ lies in a finite cone, the image is isometric to a subset of a $k$-pod, with cone point $u(0)$.  Hence, if $d(u(0),u(\phi))$ is constant for $\phi\in\Sp^1$, $u(\Sp^1)=\{p\}$ for some $p\in X$. Consider $x\mapsto d^2(u(x),u(0))$, which vanishes on $\Sp^1$. By \cite[Proposition 2.2]{gromov-schoen} the map is subharmonic and thus vanishes on $ B_1(0)$. In particular, $u$ is a constant function, contrary to our assumption.
    
    Now suppose $N \geq 2$. By Lemma \ref{lem:gammadef} and the discussion following, $\lambda = \frac{2\pi \alpha}m$. And by Lemma \ref{lem:sph-bill}, $\lambda = \frac{2\pi j }k$ where $k$ divides $|W|$. It follows that $\alpha = \frac {mj}k$ where $k$ divides $|W|$.
\end{proof}

\section{Remarks}\label{sec:rem}

In extending this proof to domain dimensions higher than $2$, several difficulties arise. We may still study homogeneous maps to resolve our problem, but these maps become more complicated. First, the analogous result to \ref{thm:flat-dichotomy} will be more involved in this context, because knowing the ``radial" behavior of a map $u$ will no longer uniquely determine the ``tangential" behavior within each apartment. Instead, for a homogeneous map $u:\R^m\to X$ of order $\alpha$, we learn that $u|\Sp^{m-1}$ is an eigenfunction of the spherical Laplacian with eigenvalue $-\alpha(\alpha+m-2)$ in neighborhoods of regular points. Moreover, there are technical aspects to be overcome---in domain dimension $\geq3$, it is possible for singular points of a homogeneous map $u$ to exist on the boundary of the domain.

To close out this note, we discuss the case where $X$ is a tree, and offer a question related to studying the possible orders of harmonic maps into trees.

In the case where $u:\Omega\to X$ is a homogeneous map into a tree, by \cite{sun}, $u$ maps into a $k$-pod for some $k$ even in the non-locally finite setting. Consider, as we have done here, the map $v:=u|\Sp^{m-1}:\Sp^{m-1}\to X$.  For each of the (open) legs of the $k$-pod, $L_1,\dots,L_k$, the preimage $S_i:=v^{-1}(L_i)$ is an open subset of $\Sp^{m-1}$, and $v(\partial S_i)=0_X$ for each $i$, where $0_X$ denotes the ``vertex" of the $k$-pod, the common boundary of all of the legs. Within each $S_i$, $v$ is an eigenfunction of the spherical Laplacian with eigenvalue $-\alpha(\alpha+m-2)$. We hence arrive at the following question.

\begin{question}
    For a fixed $m$, for which values of $\alpha\geq1$ do there exist finitely many open domains $S_1,\dots,S_k$ in $\Sp^{m-1}$ so that the following hold? \begin{enumerate}\item $\Sp^{m-1}=\bigcup_{i=1}^k\overline{S_i}$, \item $-\alpha(\alpha+m-2)$ is a Dirichlet eigenvalue of the spherical Laplacian in each $S_i$.\end{enumerate}
\end{question}

In the case $m=2$, a direct analysis of this question leads to an analogous result to Theorem \ref{thm:flat-dichotomy}---the only such $\alpha$ are $\frac m2$, for natural numbers $m\geq2$. We hence conjecture that the possible orders of harmonic maps into trees are precisely the $\alpha$ answering the previous Question. This leaves open the question of how to approach more general Euclidean buildings of dimension $>1$.

\end{document}